\documentclass[reqno,12pt,letterpaper,oneside]{amsart}
\usepackage{amsmath,amssymb,amsthm,graphicx,mathrsfs,url}
\usepackage[usenames,dvipsnames]{color}
\usepackage[colorlinks=true,linkcolor=Red,citecolor=Green]{hyperref}
\usepackage[margin=1.3in]{geometry}
\usepackage[usenames,dvipsnames]{xcolor}
\usepackage{tikz}
\usepackage{csquotes}
\usepackage[justification=centering]{caption}
\usepackage{graphicx}
\usepackage[colorlinks=true]{hyperref}
\usepackage{color}
\usepackage{amsmath,amsfonts}
\usepackage{calrsfs}
\usepackage{amsmath,environ}
\usepackage{enumitem}
\DeclareMathAlphabet{\pazocal}{OMS}{zplm}{m}{n}

\numberwithin{equation}{section}
\usepackage{amsmath, amsthm}
    \usepackage{dsfont}
    \usepackage{fancybox}
    \usepackage{amssymb}

\DeclareMathOperator{\rk}{rk}   
\DeclareMathOperator{\fl}{fl}

\newcommand{\PP}{{\mathcal{P}}}

\newcommand{\sgn}{{\text{sgn}}}

\newcommand{\FF}{\mathcal{F}}

\newcommand{\II}{\mathcal{I}}

\newcounter{quotecount}

\title[$h$-vector Inequalities under Weak Maps]{$h$-vector Inequalities under Weak Maps}
\author {Gaku Liu, Alexander Mason}

\NewEnviron{equations}{%
\begin{equation}\begin{gathered}
  \BODY
\end{gathered}\end{equation}
}

\NewEnviron{equations*}{%
\begin{equation*}\begin{gathered}
  \BODY
\end{gathered}\end{equation*}
}

\numberwithin{quotecount}{section}
\usepackage{autonum}

\newtheorem{thm}{Theorem}[section]
\newtheorem{cor}[thm]{Corollary}

\newtheorem{lem}[thm]{Lemma}
\newtheorem{proposition}[thm]{Proposition}
\newtheorem{theorem}[thm]{Theorem}

\theoremstyle{definition}
\newtheorem{definition}[thm]{Definition}

\begin{document}
\maketitle

\begin{abstract}
We study the behavior of $h$-vectors associated to matroid complexes under weak maps, or inclusions of matroid polytopes. Specifically, we show that the $h$-vector of the order complex of the lattice of flats of a matroid is component-wise non-increasing under a weak map. This result extends to the flag $h$-vector. We note that the analogous result also holds for independence complexes and rank-preserving weak maps.
\end{abstract}

\section{Introduction}

The study of matroids and their invariants has undergone remarkable developments in recent years. In particular, many long-standing conjectures such as the Heron--Rota--Welsh conjecture \cite{AHK} and the Dowling--Wilson top-heavy conjecture \cite{BHMPW} have been resolved through the development of powerful techniques. These conjectures concern inequalities that are satisfied between certain invariants, such as the number of flats of a given rank, associated to a given matroid.

In this paper we take a different perspective and consider inequalities between invariants of \emph{different} matroids. The set of all matroids admits a natural partial order whose relations are \emph{weak maps}. Intuitively, if $A$ and $B$ are matroids and $A \to B$ is a weak map, then $A$ is obtained from $B$ by perturbing $B$ to a more general position. (In terms of matroid polytopes, weak maps correspond to reverse inclusions of independence polytopes, and reverse inclusions of base polytopes if the matroids have the same rank.) Weak maps can be very complicated, even for realizable matroids: For example, weak maps of realizable matroids cannot always be realized as continuous deformations of vector configurations or as cells in a matroid subdivision. See for example \cite{Stu}.

In \cite{L}, Lucas gives many inequalities of matroid invariants under weak maps.
It is obvious that some invariants, such as the number of independent sets of given rank and the number of flats of given rank, are non-increasing under a weak map. 
Less obvious is what happens to the \emph{$h$-numbers} corresponding to these invariants. Given a vector of numbers called an $f$-vector, the $h$-vector is the image of the $f$-vector under a certain linear transformation. If the $f$-vector is the face vector of a simplicial complex, then the $h$-vector gives the numerator of the Hilbert-Poincar\'{e} series of the \emph{Stanley-Reisner ring} of the complex.

Here, we focus on two complexes in particular: the \emph{order complex of the lattice of flats} of a matroid and the \emph{independence complex} of a matroid. The \emph{lattice of flats} of a matroid $M$ is the poset whose elements are the flats of the matroid partially ordered by containment. A lattice isomorphic to the lattice of flats of some matroid is also called a \emph{geometric lattice}. The \emph{order complex} of this poset is the simplicial complex whose simplices are the chains of the poset. We denote this complex by $\Delta(M)$. This complex is essentially the same as the Bergman fan of the matroid. The \emph{independence complex} of a matroid is the simplicial complex whose simplices are the independent sets of the matroid. We denote this complex by $\Delta_I(M)$. 

Our main result is the following.

\begin{thm} \label{summarythm}
Let $A$ and $B$ be matroids and $A \to B$ a weak map. The following are true.
\begin{enumerate}
\item The $h$-vector of $\Delta(A)$ is component-wise at least the $h$-vector of $\Delta(B)$. 
\item If $A$ and $B$ have the same rank, then the $h$-vector of $\Delta_I(A)$ is component-wise at least the $h$-vector of $\Delta_I(B)$.
\end{enumerate}
\end{thm}

We note that (2) is an immediate consequence of Stanley's \cite{S1} monotonicity theorem on injections of simplicial complexes. Therefore the paper is mainly devoted to proving (1). We observe that a weak map of matroids induces a surjection of the corresponding geometric lattices, but this surjectivity alone is not enough to imply the result for general lattices (or even geometric lattices), so (1) is a special property of geometric lattices and weak maps.

Our result for (1) is actually finer, and holds for \emph{flag $h$-vectors}. The flag $h$-vector of a graded poset of rank $r$ is a certain vector $(h_S : S \subset \{1,\dots,r\})$ with the property that $\sum_{|S| = k} h_S$ is equal to $h_k$ of the order complex of the poset.
We prove the following:

\begin{thm} \label{mainthm}
Let $A$ and $B$ be matroids of rank $r$ and $A \to B$ a weak map. Then the flag $h$-vector of $\Delta(A)$ is component-wise at least the flag $h$-vector of $\Delta(B)$. 
\end{thm}

These results can be interpreted in terms of valuative invariants. The (flag) $h$-vectors associated to the order complex of the lattice of flats and independence complex of a matroid are known to be valuative invariants of the matroid (the fact that the flag $f$-vector of the lattice of flats is valuative was recently proven in \cite{FS}). Our results can be interpreted as saying that these invariants are monotonic with respect to inclusion of matroid polytopes. It would be interesting to determine which other well-known valuative invariants have this property.  From \cite{L}, such invariants include the M\"{o}bius function, Crapo's beta invariant, and Whitney numbers of the first kind.

Our work is inspired by previous work of Nyman and Swartz \cite{NS}, where they find the component-wise maximizers and minimizers of the flag $h$-vector of $\Delta(M)$ over all matroids of fixed rank and size. In particular, the flag $h$-vector is maximized by the uniform matroid and minimized by the near-pencil matroid. All matroids have a weak map from a uniform matroid of the same size and rank, so our result recovers their maximizer. On the other hand, not all matroids have a weak map to the near-pencil matroid of the same size and rank, and in general the minimal matroids of given size and rank with respect to the weak map order is not well-understood.

The proof idea is as follows. Given a weak map of matroids $A \to B$, we construct a degree-preserving map from the Stanley--Reisner ring of $B$ to a certain quotient of the Stanley--Reisner ring of $A$. This map is readily seen to be injective, but it is much harder to show that the map remains injective after quotienting by a linear system of parameters. We do this in an indirect way, by showing that the dual map between the corresponding dual vector spaces is surjective.

\section{Matroid Preliminaries}

In this section we establish terminology and notation. We will assume the reader is already familiar with the basic properties of matroids and refer to \cite{Ox} for further background.

\begin{definition}
A matroid $M$ is a (finite) ground set $E$ together with a collection $\II(M)$ of subsets called \emph{independent sets}. They have the following properties:

\begin{enumerate}
\item A subset of an independent set is independent.

\item Given two independent sets with $|A| < |B|$, there is some $x \in B \setminus A$ such that $x \cup A$ is also independent.

\item $\emptyset$ is independent.
\end{enumerate}
\end{definition}

In this paper we assume all matroids have the same ground set $[n] = \{1, \dots , n\}$.


\begin{definition}
A \emph{flat} of a matroid $M$ is $F \subseteq E$ such that if $I$ is an independent subset of $F$ and $x \in E \setminus F$, then $I \cup \{x\}$ is independent.
\end{definition}

Write $\mathcal{F}(M)$ for the set of flats of $M$. Note that $\mathcal{I}(M)$ and $\mathcal{F}(M)$ are both posets ordered by inclusion. $\mathcal{F}(M)$ is a lattice called ``the lattice of flats of $M$".

\begin{proposition} \label{flatsdef}
Flats have the following properties:
\begin{enumerate}
\item An intersection of two flats is a flat.

\item Given a flat $F$ and $x \in E \setminus F$, there is a unique flat $G$ containing $x$ that covers $F$ in the poset $\mathcal{F}(M)$.

\item $E$ is a flat.
\end{enumerate}
\end{proposition}

\begin{definition}
The \emph{rank} of a set $G \subseteq E$ is the size of the largest independent set it contains, or, equivalently, that of the smallest flat containing it. We denote the rank of $G$ by $\rk(G)$. The rank of the matroid $M$ is defined to be $\rk(E)$. An independent set of size $\rk(E)$ is called a \emph{basis}.
\end{definition}


\begin{definition} \label{def:closure}
The \emph{closure map} $\phi_M : \PP(E) \rightarrow \FF(M)$ (where $\PP(E)$ is the power set of $E$) is defined so that $\phi_M (G)$ is the smallest flat containing $G$.
\end{definition}

The subscript of $\phi_M$ may be omitted when it is unambiguous which matroid is being referred to. If $\phi_M (G) = F$, then we say that $F$ is the \emph{closure} (or ``$M$-closure") of $G$, or that $G$ \emph{spans} $F$.

The map $\phi$ preserves containment: If $G \subseteq G'$, then $\phi(G) \subseteq \phi(G')$.

\begin{definition}
The \emph{order complex} of a poset is the simplicial complex whose faces are the chains of the poset.
\end{definition}

Let $0_M$ denote the minimal flat of a matroid $M$. (In other words, $0_M$ is the set of all loops of $M$, and is $\emptyset$ if $M$ is loopless.) We will write $\Delta(M)$ for the order complex of $\FF(M) \setminus \{0_M, E\}$.

\section{The $f$- and $h$-vectors} \label{sec:fh}

The usual $f$- and $h$-vectors for a simplicial complex are defined as follows:

\begin{definition}
Let $\Delta$ be a simplicial complex of dimension $r-1$.
\begin{enumerate}

\item The \emph{$f$-vector} of $\Delta$ is the sequence $(f_i (\Delta))_{i=0}^r$, where $f_i (\Delta)$ is the number of faces with cardinality $i$.\footnote{For convenience, we index the $f$-vector by cardinality instead of dimension. We will then modify the definition of the $h$-vector so that it agrees with the usual indexing of the $h$-vector.}

\item The \emph{$h$-vector} is the sequence $(h_i (\Delta))_{i=0}^r$ satisfying
\[
\sum_{i = 0} ^r h_i x^{r-i} = \sum_{i = 0} ^r f_i (x-1)^{r-i}.
\]

\end{enumerate}
\end{definition}

We now give a refinement of the $f$- and $h$-vectors for posets.
Let $P$ be a (finite) graded poset with rank function $\rk$. We define the \emph{rank} of $P$ to be the maximal cardinality of a chain. Given a chain $C$ in $P$, the \emph{flag} of $C$, written $\fl(C)$, is the set of ranks of flats in that chain. That is, $\fl(C) = \{\rk(F)\}_{F \in C}$. This is a subset of $[r]$, where $r$ is the rank of the poset. For our purposes, the empty set will also be considered a chain, with flag $\emptyset$.

\begin{definition}
Let $P$ be a graded poset of rank $r$.
\begin{enumerate}

\item The flag $f$-\emph{vector} of $P$ is the tuple $(f_S (P))$ taken over all $S \in \PP ([r])$, where $f_S (P)$ is the number of chains $C$ such that $\fl C = S$.

\item The flag $h$-\emph{vector} is the tuple $(h_S (P))$ taken over all $S \subseteq \PP ([r])$, where
\[
h_S = \sum_{T \subseteq S} (-1)^{|S| - |T|} f_T.
\]
\end{enumerate}
\end{definition}

While the flag vectors are defined for posets, we will abuse notation and say that $(f_S(P))$ is the flag $f$-vector of the order complex $\Delta(P)$.

We write $f_S (M)$ for $f_S (\Delta(M))$, $f_k (M)$ for $f_k (\Delta(M))$, and do similarly for the $h$-vectors.

\begin{proposition} \label{fh-vectors}
The following are true.
\begin{enumerate}
\item $f_k = \sum_{|S|=k} f_S$.

\item $h_k = \sum_{|S|=k} h_S$.

\item $f_S = \sum_{T \subseteq S} h_T$.

\item $f_{r} = \sum_{S \subseteq [r]} h_S$.
\end{enumerate}
\end{proposition}

To avoid confusion with other notions of minimality and maximality, chains of maximal length (i.e. chains of flag $[r]$) will be called \emph{full} in this paper.

We now focus on the case when $\Delta = \Delta(M)$ where $M$ is a matroid. There is a useful partition of the set of full chains of $\Delta(M)$ into sets of sizes $h_S$, as follows: Given a full chain of flats $0_M \varsubsetneq F_1 \varsubsetneq \dots  \varsubsetneq F_r \varsubsetneq E$, let $b_i = \min F_i \setminus F_{i-1}$, where elements of the ground set $[n]$ are ordered in the usual way. (Here, $F_0 = 0_M$ and $F_{r+1} = E$.) The resulting string $b_1 \dots b_{r+1}$ is sometimes called the chain's \emph{Jordan-H\"{o}lder sequence} \cite{B1}. By property (2) of Prop~\ref{flatsdef}, any element of $F_i \setminus F_{i-1}$ determines $F_i$ given $F_{i-1}$. Thus there is an injection from full chains of flats to ordered sets of size $r+1$ in $[n]$.
 
Note that not all such ordered sets come from chains of flats. First, each $b_i$ must not be in the flat spanned by $b_1, \dots, b_{i-1}$, or, equivalently, $\{b_1, \dots, b_{r+1}\}$ must be a basis for $E$. However, each $b_i$ must also be the minimal element in the uniquely determined $F_i \setminus F_{i-1}$. Call an ordered basis that has this latter property, and thus corresponds to a chain of flats, ``valid".

Now given a string $b_1 \dots b_{r+1}$, we say the string (or its corresponding full chain, if it has one) has a descent \emph{across} position $i$ (or alternatively, across the corresponding $F_i$) if $b_i > b_{i+1}$, and that it has an \emph{ascent} otherwise. The set of all indices across which a string (full chain) has a descent is that string's \emph{descent set}.

\begin{theorem} \label{hvenum} \cite{S2}
Let $M$ be a matroid of rank $r+1$.

\begin{enumerate}
\item The set of valid strings with descent sets contained in $S \subseteq [r]$ has cardinality $f_S (M)$.

\item The set of valid strings with descent sets equal to $S \subseteq [r]$ has cardinality $h_S (M)$.
\end{enumerate}
\end{theorem}

Because we will use some of the constructions from the proof later, we provide a proof of this theorem.

\begin{proof} (1) Fix $S \subseteq [r]$. We will demonstrate a bijection between chains of flag $S$ and full chains with descent set contained in $S$. First, note that any non-full chain of flats $0_M = F_0 \varsubsetneq F_1 \varsubsetneq \dots \varsubsetneq F_k = E$ has a unique \emph{minimal completion} to a full chain as follows: for each interval $[F_i,F_{i+1}]$ where $\rk(F_{i+1}) > \rk(F_i) + 1$, let $F_{i, 1}$ be the flat covering $F_i$ containing $a_{i, 1} := \min(F_{i+1} \setminus F_i)$. Then, inductively, let $F_{i, j+1}$ be the flat covering $F_{i, j}$ containing $a_{i, j+1} := \min(F_{i+1} \setminus F_{i, j})$ for $1 \leq j \leq k_i$, where $k_i = \rk(F_i) - \rk(F_{i+1})-1$.

Now consider the full chain
\[
0_M \varsubsetneq F_{0, 1} \varsubsetneq \dots \varsubsetneq F_{0, k_0} \varsubsetneq F_1 \varsubsetneq F_{1, 1} \varsubsetneq \dots \varsubsetneq F_k = E.
\]
By construction, $b_{i+1} = a_{i, 1} < a_{i, 2} < \dots < a_{i, k_i} < \min(F_{i+1} \setminus F_{i, k_i})$, so each new flat $F_{i, j}$ has an ascent across it. That is, the descent set of this chain is contained in $S$.

Denote by $\mu(C)$ the minimal completion of a chain $C$ of flag $S$. Let $\nu$ be the map that restricts a full chain to the flats with ranks in $S$. We claim that $\nu$ is the inverse of $\mu$. Clearly $\nu(\mu(C)) = C$. To show $\mu(\nu(C)) = C$, it suffices to check that the minimal completion is unique, in that $\mu(C)$ is the unique full chain containing $C$ with no descents outside $S$. Suppose that $D$ is some other full chain containing $C$, and let $G_i$ be its flat of rank $i$. Let $G_j$ be the first flat in which $D$ differs from $\mu(C)$, and $F_i, F_{i+1}$ the flats of $C$ such that $\rk(F_i) < j < \rk(F_{i+1})$. Then $G_{j-1}$ was constructed by the above interpolation process, while $G_j$ was not. That is, $\min(F_{i+1} \setminus G_{j-1}) \notin G_j$. Let $a = \min(F_{i+1} \setminus G_{j-1})$, and let $G_k$ be the first flat in $D$ which contains $a$. Then $\min(G_{k-1} \setminus G_{k-2}) > a = \min(G_{k} \setminus G_{k-1})$, so $G_{k-1}$ is a flat of $D$ with a descent across it, whose rank is not in $S$ since $G_{k-1} \in [F_i, F_{i+1}]$ but is neither $F_i$ nor $F_{i+1}$. This proves the uniqueness of $\mu(C)$.

Therefore $f_S (M)$, which counts the number of chains of flag $S$, also counts the number of valid strings with descent set $\subseteq S$.

(2) immediately follows from the identity $f_S = \sum_{T \subseteq S} h_T$.
\end{proof}

A flat $F$ is \emph{minimal} in a (poset) interval $[G, H]$ if $F$ is one of the flats generated by the interpolation process described in the above proof. That is, if $\rk(F) = \rk(G) + j$, then $F$ contains the successively minimal elements $a_{i, 1}, a_{i, 2}, \dots, a_{i, j}$ found in the inductive process described above for $F_i = G$, $F_{i+1}=H$. For any interval, there is exactly one minimal flat of each rank.

A chain of flats $0_M = F_0 \varsubsetneq F_1 \varsubsetneq \dots \varsubsetneq F_k = E$ is \emph{nonessential} if it contains at least one flat that is minimal with respect to its neighbors, that is, some $F_i$ that is minimal in $[F_{i-1}, F_{i+1}]$. Otherwise, the chain is \emph{essential}. Thus we can rephrase the above result as follows:

\begin{proposition}
$h_S (M)$ counts the number of essential chains of flag $S$.
\end{proposition}

\begin{proof}
From the proof of Thm.~\ref{hvenum}, a chain of flag $S$ is essential if and only if the descent set of its unique minimal completion is $S$.
\end{proof}

Finally, note that since $h_S$ counts the number of valid strings with descent set $S$, $h_k$ counts the number of valid strings with exactly $k$ descents.

\section{Strong and Weak Maps of Matroids}

We now define strong and weak maps of matroids. We refer to \cite[Chapter 8]{Wh} for more information.

\begin{definition}

Let $A$, $B$ be two matroids on the same ground set $E = [n]$.

\begin{enumerate}

\item There is a \emph{strong map} from $A$ to $B$ if $\FF(B) \subseteq \FF(A)$.

\item There is a \emph{weak map} from $A$ to $B$ if $\II(B) \subseteq \II(A)$.
\end{enumerate}
\end{definition}

First, we give a result which explains the nomenclature:

\begin{proposition} 
All strong maps are weak maps.
\end{proposition}

Note that in neither case are the ranks of $A$ and $B$ assumed to be equal, although we clearly have $\rk(A) \geq \rk(B)$, since a basis for $B$ is independent in $A$. It will often be useful to restrict to the case of rank-preserving weak maps, i.e. those where $\rk(A) = \rk(B)$. In contrast, a strong map can only have $\rk(A) = \rk(B)$ if $A = B$.

When there is a weak map $A \to B$, we will often consider the map $\phi_B: \FF(A) \to \FF(B)$ which is the restriction of the closure map $\phi_B : \PP(B) \to \FF(B)$ from Definition~\ref{def:closure}.
There is a natural extension of this map to chains.

\begin{definition}
If $A \to B$ is a weak map, denote by $\phi_B : \Delta(A) \to \Delta(B)$ the map defined as follows. Given $C = (0_A \varsubsetneq F_1 \varsubsetneq \dots \varsubsetneq F_k \varsubsetneq E)$ a chain in $A$, take $(0_B \subseteq \phi_B (F_1) \subseteq \dots \subseteq \phi_B (F_k) \subseteq E)$, then delete any duplicate flats to obtain $\phi_B (C)$.
\end{definition}

We first observe the following.

\begin{thm} \label{smineq}
If $A \to B$ is a strong map, then $h_k (A) \geq h_k (B)$ for all $0 \leq k \leq \rk(B) - 1$.
\end{thm}

For the proof of this, we use the following monotonicity theorem by Stanley:

\begin{thm}[\cite{S1}] \label{thm:mono}
Let $\Delta'$ be a subcomplex of the simplicial complex $\Delta$, where both are Cohen-Macaulay. Suppose that $e-1 = \dim \Delta' \leq \dim \Delta = d-1$, and that no set of $e+1$ vertices of $\Delta'$ form a face of $\Delta$. Then $h_k (\Delta') \leq h_k (\Delta)$ for all $k$.
\end{thm}

\begin{cor} \label{cor:mono}
Let $P$ be a poset and $Q$ be an induced subposet of $P$ such that $\Delta(P)$ and $\Delta(Q)$ are both Cohen-Macaulay. Then $h_k(\Delta(Q)) \le h_k(\Delta(P))$ for all $k$.
\end{cor}

\begin{proof}
Since $Q$ is an induced subposet of $P$, any subset of $Q$ which is not a chain in $Q$ is also not a chain in $P$, so Theorem~\ref{thm:mono} applies.
\end{proof}

See Section~\ref{sec:SR} for the definition of Cohen-Macaulay. In particular, order complexes of lattices of flats of matroids are shellable and thus Cohen-Macaulay \cite{B2}.

\begin{proof}[Proof of Theorem~\ref{smineq}]
We have that $\FF(B)$ is a subposet of $\FF(A)$ by definition of a strong map, and it is an induced subposet because the relation in both posets is set inclusion. Thus Corollary~\ref{cor:mono} gives the result.
\end{proof}

Note that since every $f_k$ is a nonnegative sum of the $h_j$, we have the following.

\begin{cor}
If $A \to B$ is a strong map, then $f_k (A) \geq f_k (B)$ for all $0 \leq k \leq \rk(B)$.
\end{cor}

We now consider weak maps and flag vectors. Since flag vectors are defined in terms of the rank of the matroid, there are few comparisons that can be made unless the two matroids have the same rank. Therefore, we consider only rank-preserving weak maps. The following proposition (together with the already-discussed case of strong maps) shows it suffices to consider only this case.

\begin{proposition} \label{mapdecomp} 
Every weak map of matroids can be decomposed into a strong map and a rank-preserving weak map.
\end{proposition}

We will prove one more combinatorial result. We first prove the following lemma:

\begin{lem} \label{phidecr}
Let $A \rightarrow B$ be a weak map.

\begin{enumerate}
\item If $G \subseteq E$, then $\rk_A (G) \geq \rk_B (G)$.

\item $\phi_B : \FF(A) \rightarrow \FF(B)$ does not increase a flat's rank.
\end{enumerate}
\end{lem}

\begin{proof}
(1) Let $I$ be a $B$-basis for $G$. $I$ is also independent in $A$, so the largest $A$-independent set contained in $G$ is at least size $|I|$.

(2) Note that for all $G \subseteq E, \rk_B(G) = \rk_B (\phi_B (G))$. Substituting this into the inequality from (1) gives the result.
\end{proof}

We now show the following:

\begin{proposition} \label{phisurj}
Let $A \rightarrow B$ be a rank-preserving weak map. Then,

\begin{enumerate}
\item $\phi_B: \mathcal{F} (A) \rightarrow \mathcal{F} (B)$ is surjective.

\item $\phi_B: \mathcal{F} (A) \rightarrow \mathcal{F} (B)$ is surjective by rank: Given $F \in \mathcal{F}(B)$, there exists $G \in \mathcal{F}(A)$ with $\phi_B (G) = F$ and $\rk(G) = \rk(F)$.

\item $\phi_B: \Delta(A) \to \Delta(B)$ is surjective by flag: Given $C \in \Delta(B)$, there exists $D \in \Delta(A)$ with $\phi_B (D) = C$ and $\fl(D) = \fl(C)$.
\end{enumerate}
\end{proposition}

\begin{proof}
It suffices to prove (3). Let $C \in \Delta(B)$ be of flag $S$, and let $b_1 \dots b_{r+1}$ be the ordered basis for its minimal completion. That is, if $k \in S$, then $F_k := \phi_B (\{b_1, \dots, b_k\})$ is the rank-$k$ flat in $C$. Now let $D$ be the chain of flag $S$ whose rank-$k$ flat is $G_k := \phi_A (\{b_1, \dots, b_k\})$. We know $\phi_B (G_k)$ is a $B$-flat which contains $\{b_1, \dots, b_k\}$, and therefore $\phi_B (G_k) \supseteq F_k$. However, by the lemma, $\rk_B (\phi_B (G_k)) \leq rk_A (G_k) = k$. Thus $\phi_B (G_k) = F_k$.
\end{proof}

(3) immediately implies the following:

\begin{cor}
\
\begin{enumerate}
\item If $A \to B$ is a rank-preserving weak map with both matroids of rank $r+1$, then $f_S (A) \geq f_S (B) \ \forall S \subseteq [r]$.

\item If $A \rightarrow B$ is a weak map, then $f_k (A) \geq f_k (B)$ for all $k$.
\end{enumerate}
\end{cor}

We will strengthen this result in Thm.~\ref{finalthm}.

\section{The Independence Complex}

Before proceeding further on lattices of flats, we consider the independence complex.

\begin{definition}
The \emph{independence complex} $\Delta_I (M)$ of a matroid $M$ is the simplicial complex whose faces are given by $\mathcal{I}(M)$.
\end{definition}

Denote the $f$- and $h$-vectors of $\Delta_I (M)$ by $h^I (M)$ and $f^I (M)$ respectively.

One immediate consequence of the definitions is that if $A \to B$ is a weak map, then the identity map provides an injection of $\mathcal{I} (B)$ into $\mathcal{I} (A)$, which in turn implies that $f_k ^I (A) \geq f_k ^I (B)$ for all $k$. However, we also have the following stronger result.



\begin{proposition}
If $A \to B$ is a rank-preserving weak map, then $h_k ^I (A) \geq h_k ^I (B)$ for all $k$.
\end{proposition}

\begin{proof}
By definition of rank-preserving weak map, $\Delta_I (B)$ is a subcomplex of $\Delta_I (A)$ and they have the same dimension. In addition, independence complexes of matroids are shellable and thus Cohen-Macaulay \cite{B2}. Hence, the result follows from Theorem~\ref{thm:mono}.
\end{proof}

The statement is not true for strong maps (or weak maps that change rank). For example, let $A$ be the rank 2 uniform matroid on 2 elements and $B$ the rank 1 uniform matroid on 2 elements. Then we have a strong map $A \to B$ but $h_1^I(A) = 0$, $h_1^I(B) = 1$.

\section{The Stanley-Reisner Ring} \label{sec:SR}

Fix an infinite field $k$. A simplicial complex $\Delta$ has an associated ring $k[\Delta]$, called the \emph{Stanley-Reisner ring} of the complex: $k[\Delta] = k[x_{v_1}, \dots , x_{v_m}]/I_\Delta$, where $\{v_j\}$ is the set of vertices of the complex, and $I_\Delta$ is the ideal generated by monomials of the form $x_{v_{j_1}} \dotsb x_{v_{j_k}}$, where $\{v_{j_1}, \dots, v_{j_k}\}$ is not a face of the complex. Note that $k[\Delta]$ is graded by degree, which we call the ``coarse'' grading.

Let $P$ be a graded poset of rank $r$ and $\Delta$ its order complex. Then $k[\Delta]$ has an $\mathbb{N}^r$-grading defined as follows: Let $v_1 < \dots < v_r$ be a full chain in $P$ and $d_1$, \dots, $d_r$ nonnegative integers. Then the degree of (the image of) the monomial $x_{v_1}^{d_1} \dotsb x_{v_r}^{d_r}$ in $k[\Delta]$ is defined to be $(d_1,\dots,d_r)$. We call this the ``fine'' grading of $k[\Delta]$. The fine graded component of $k[\Delta]$ corresponding to a tuple $\alpha$ will be denoted $k[\Delta]_\alpha$. If $d_i = 0$ or 1 for all $1 \le i \le r$, then we say $x_{v_1}^{d_1} \dotsb x_{v_r}^{d_r}$ has degree $S$, where $S = \{i \mid d_i = 1\}$, and analogously define $k[\Delta]_S$.

In the case we are interested in, where $\Delta = \Delta(M)$ for a matroid $M$, the polynomial ring is generated by variables indexed by the flats (aside from the empty flat and $E$), and $I_\Delta$ is generated by products of any two variables corresponding to incomparable flats. In particular, $I_\Delta$ includes all monomials except those that are the product of variables from a chain. We will write $k[M]$ for $k[\Delta(M)]$.
Given a chain of flats $C$, write
\[
x_C = \prod_{F \in C} x_F.
\]
An arbitrary element $p \in k[M]_S$ can be expressed as $\sum_{\fl(C) = S} a_C x_C$ with $a_C \in k$.


Let $A$ be a finitely generated graded $k$-algebra. A \emph{system of parameters} for $A$ is a sequence $\theta_1$, \dots, $\theta_r \in A$ of minimal length such that $A / \langle \theta_1,\dots,\theta_r \rangle$ is finite-dimensional over $k$. A system of parameters is \emph{homogeneous} if all of its elements are homogeneous (with respect to the coarse or fine grading, depending on context), and \emph{linear} if all of its elements have coarse degree 1. Assuming $k$ is infinite, there always exists a homogeneous linear system of parameters.

A \emph{regular sequence} in $A$ is a sequence $\theta_1$, \dots, $\theta_r \in R$ such that $\theta_{i}$ is not a zero-divisor in $A / \langle \theta_1,\dots,\theta_{i-1} \rangle$ for all $1 \le i \le r$. We say that $A$ is \emph{Cohen-Macaulay} if every system of parameters of $A$ is a regular sequence. The significance of this definition in combinatorics is the following observation:

\begin{thm}
Let $\Delta$ be a simplicial complex and assume $k[\Delta]$ is Cohen-Macaulay. Let $\theta_1$, \dots, $\theta_r$ be a linear system of parameters and let $R = k[\Delta] / \langle \theta_1,\dots,\theta_r \rangle$, which inherits the coarse grading from $k[\Delta]$. Then for all $i$, the dimension of the degree-$i$ component of $R$ is $h_i(\Delta)$.
\end{thm}

Now assume $\Delta$ is the order complex of a graded poset $P$ of rank $r$. In this case, there is a particularly nice linear system of parameters for $k[\Delta]$, given by
\[
\theta_i = \sum_{v \in P,\, \rk v = i} x_v
\]
for $i = 1$, \dots, $r$. Note that this system of parameters is homogeneous with respect to the fine grading.

\begin{theorem}[\cite{S0}] \label{thm:CMflag}
Let $\Delta$ be the order complex of a graded poset $P$ and assume $k[\Delta]$ is Cohen-Macaulay. Let $(\theta_i)$ be as above, and let $R = k[\Delta] / \langle \theta_1,\dots,\theta_r \rangle$, which inherits the fine grading from $k[\Delta]$. The following are true. 
\begin{enumerate}
\item The dimension of the degree-$S$ component of $R$ is $h_S (P)$.
\item If $\alpha = (d_1,\dots,d_r)$ and $d_i > 1$ for any $i$, then $(R_M)_\alpha = 0$.
\end{enumerate}
\end{theorem}

We now focus on the case where $\Delta = \Delta(M)$ for a matroid $M$ of rank $r+1$.
Set $\theta_i = \sum_{F \in \FF, \, \rk F = i} \ x_F$ as above, and let $\Theta_M$ to be the ideal generated by the $\theta_i$ over $i \in [r]$. (We may drop the subscript $M$ if it is clear.) Define $R_M = k[M]/\Theta_M$. By the above theorem, $\dim (R_M)_S = h_S(M)$.

We make one more definition before moving on.

\begin{definition}
Let $M$ be a matroid.
\begin{enumerate}
\item The \emph{lexicographic order} on rank $k$ flats in $\FF(M)$ is defined as follows: given two flats $F \neq G$, let $j$ be the first element of the ground set $[n]$ contained in one of $F$, $G$ but not the other; if $j \in F$ but $j \notin G$, we say that $F < G$.

\item The \emph{lexicographic order} on flag $S$ chains in $\Delta(M)$ is defined as follows: given two chains $C = \{F_i\}$, $C' = \{G_i\}$, let $k$ be the lowest rank such that $F_k \neq G_k$; if $F_k < G_k$, we say that $C < C'$.
\end{enumerate}
\end{definition}

Note that this order is consistent with the notion of minimality in Section~\ref{sec:fh}: if two flats of the same rank $G$, $G'$ are both contained in the interval $[F, H]$, and $G$ is minimal with respect to that interval, then $G \leq G'$.

\section{Matroid maps and ring maps}

Let $\Delta$, $\Delta'$ be simplicial complexes, and let $\phi : \Delta \to \Delta'$ be a map of complexes (that is, $f(\sigma) \le f(\tau)$ for all $\sigma$, $\tau \in \Delta$ such that $\sigma \le \tau$). Then we have a map $\psi : k[\Delta'] \to k[\Delta]$ defined by
\begin{equation} \label{eq:map}
\psi(x_\sigma) = \sum_{\tau \in \phi^{-1}(\sigma)} x_\tau
\end{equation}
for all $\sigma \in \Delta$. (Here, $x_\sigma = \prod_{v \in \sigma} x_v$.) It is straightforward to check that this gives a well-defined homomorphism $k[\Delta'] \to k[\Delta]$. Moreover, we have the following.

\begin{proposition}
$\phi$ is surjective if and only if $\psi$ is injective.
\end{proposition}

\begin{proof}
It is clear from the definition that $\psi$ is injective if and only if $\psi(p) \neq 0$ for all monomials $p \in k[\Delta']$. This is easily seen to be equivalent to $\phi$ being surjective.
\end{proof}

Given a rank-preserving weak map of matroids $A \to B$, Prop.~\ref{phisurj} says we have a surjective map $\phi_B : \Delta(A) \to \Delta(B)$, and thus we have an injective map $\psi_B : k[B] \to k[A]$. However, this map does not preserve the fine grading of $k[A]$ and $k[B]$, as $\phi_B$ may decrease the rank of some flats. To rectify this, we introduce a new ring $k[A']$.

Given two matroids $A$, $B$ on the same ground set $E = [n]$, we define the \emph{auxiliary pseudo-matroid} $A'$ to be the ground set $E$, together with the set $\FF(A')$ of all flats $F \in \FF(A)$ such that $\rk_B(\phi_B (F)) = \rk_A(F)$. Equivalently, a flat of $A$ is in $\FF(A')$ if and only if it has a basis which is independent in $B$. We call $\FF(A')$ the flats of $A'$, although $A'$ is not necessarily a matroid.

\begin{proposition}
If $A \rightarrow B$ is a rank-preserving weak map, then $\FF(A')$ is graded by $\rk_A.$
\end{proposition}

\begin{proof}
What we need to show is if $F$, $F' \in \FF(A')$ and $\rk_A(F') > \rk_A(F) +1$, then there exists $G \in \FF(A')$ with $F \varsubsetneq G \varsubsetneq F'$. Now $\phi_B (F) \varsubsetneq \phi_B (F')$ since both flats maintain their ranks under $\phi_B$, and this in turn implies $F' \nsubseteq \phi_B (F)$. Let $x \in F' \setminus \phi_B (F)$, and $G = \phi_A (F \cup \{x\})$. Then $\rk_A(G) = \rk_A (F) + 1$, so $F \varsubsetneq G \varsubsetneq F'$. By Lemma~\ref{phidecr}, $\rk_B (\phi_B (G))$ is either $\rk_B (\phi_B (F))$ or $\rk_B (\phi_B (F)) +1$. It cannot be the former, since then we would have $\phi_B (G) = \phi_B (F)$, but $\phi_B (G) \ni x \notin \phi_B (F)$. Thus $G \in \FF(A').$
\end{proof}

We define $\Delta(A')$ to be the order complex of $\FF(A') \setminus \{0_A,E\}$
and let $k[A']$ be the Stanley-Reisner ring of $\Delta(A')$. Since $A'$ is graded, $k[A']$ is fine-graded in the sense of Section~\ref{sec:SR}. Note that $k[A']$ is not necessarily Cohen-Macaulay.  Since the restriction of the closure map $\phi_B : \FF(A') \rightarrow \FF(B)$ preserves containment, the induced map on chains $\Delta(A') \rightarrow \Delta(B)$ also preserves flag. Thus we have a homomorphism $\psi^A _B : k[B] \to k[A']$ as in \eqref{eq:map}. (This map will usually just be written $\psi$.) This map preserves the fine grading of $k[A']$ and $k[B]$, and by Prop.~\ref{phisurj}(3) it is injective.

Analogously to matroids, define $\theta_i \in k[A']$ as $\sum x_F$ taken over all $F \in \FF(A')$ with $\rk(F) = i$, and let $\Theta_{A'}$ be the ideal generated by the $\theta_i$. Define $R_{A'} = k[A']/\Theta_{A'}$, which inherits the fine grading of $k[A']$. Since $\psi^{A} _B (\Theta_B) \subseteq \Theta_{A'}$, $\psi$ induces a well defined map $\bar{\psi}$ from $R_B$ to $R_{A'}$. The following result shows why we can work with $A'$ instead of $A$.

\begin{proposition} \label{a'reduce}
If $A \rightarrow B$ is a rank-preserving weak map and $A'$ its auxiliary pseudo-matroid, then $\dim (R_A)_S \geq \dim (R_{A'})_S$ for all $S$. In particular, $R_{A'}$ is finite-dimensional over $k$.
\end{proposition}

\begin{proof}
$k[A']$ is equal to $k[A]/J$, where $J$ is the ideal generated by all $x_F$ such that $\rk(\phi_B (F)) \neq \rk(F)$. The induced map $R_A \rightarrow R_{A'}$ is fine degree-preserving and surjective, since the composition $k[A] \to k[A'] \to R_{A'}$ is surjective.
\end{proof}

\begin{cor} \label{cor:reduction}
Let $A \to B$ be a rank-preserving weak map such that $\bar{\psi}: R_B \rightarrow R_{A'}$ is injective. Then $h_S (A) \geq h_S (B)$ for all $S$.
\end{cor}

\begin{proof}
The hypothesis is equivalent to the statement that the restriction of $\bar{\psi}$ to degree $S$ is injective for all $S$. Then by Thm.~\ref{thm:CMflag} and Prop.~\ref{a'reduce},
\[
h_S (A) = \dim_k ((R_A)_S) \geq \dim_k ((R_{A'})_S) \geq \dim_k ((R_B)_S) = h_S (B).
\]
\end{proof}


Thus, we have reduced the statement that $h_S (A) \geq h_S (B)$ for a rank-preserving weak map $A \rightarrow B$ to the following claim: Let $A \rightarrow B$ be a rank-preserving weak map of matroids. Then the map $\bar{\psi}: R_B \rightarrow R_{A'}$ is injective.

\section{Proof of the Main Theorem}

In this section we prove our main result, Thm.~\ref{mainthm}. As stated earlier, we prove that if $A \to B$ is a rank-preserving weak map, then $\bar{\psi}$ is injective. We do this by showing the induced map of the dual vector spaces is surjective, by finding preimages for each element of a basis.

Let $M$ be a matroid of rank $r+1$. Let $k[M]^\ast$ denote the dual vector space to $k[M]$, and let $\Phi_M \subseteq k[M]^\ast$ be the annihilator of $\Theta_M \subseteq k[M]$. We have $\Phi_M = \bigoplus_{S \subseteq [r]} (\Phi_M)_S$, where $(\Phi_M)_S$ can be identified as the space of linear functionals on $k[M]_S$ which annihilate $\Theta_S$. 


Given a chain $C \in \Delta(M)$, let $\epsilon_C \in k[M]^\ast$ be the functional satisfying $\epsilon_C (x_C) = 1$ and $\epsilon_C (x_D) = 0$ for all $D \neq C$. Thus an arbitrary element of $k[M]_S^*$ can be written as $\sum_{\fl(C) = S} b_C \epsilon_C$ where $b_C \in k$ for all $C$.

\begin{proposition} \label{phicond}
A functional $f = \sum_{\fl(C) = S} b_C \epsilon_C$ lies in $(\Phi_M)_S$ if and only if for all $i \in S$ and all chains $C$ with $\fl(C) = S\setminus i$, we have $\sum_{D \supseteq C} b_D = 0$.
\end{proposition}

\begin{proof}
The latter condition is satisfied if and only if $f$ annihilates all elements of the form $\theta_{i} x_C$ with $\fl(C)=S \setminus i$. Since these elements generate $\Theta_S$, the result follows.
\end{proof}

Now let $A \to B$ be a rank-preserving weak map, where $\rk(A) = \rk(B) = r+1$. Let $\pi : k[A']^\ast \to k[B]^\ast$ be the map dual to $\psi$ (i.e., it is defined by pre-composition with $\psi$). For each $S \subseteq [r]$, we can also view $\pi$ as a map $k[A']_S ^* \to k[B]_S ^*$.


\begin{theorem}
If $A \to B$ is a rank-preserving weak map of matroids, and $\pi: k[A']^* \to k[B]^*$, as well as the vector subspaces $\Phi_{A'}$ and $\Phi_B$, are as defined above, then $\pi$ maps $\Phi_{A'}$ surjectively onto $\Phi_B$. 
\end{theorem}

\begin{proof}
We begin with the following observation.

\begin{lem}
The dimension of the degree-$S$ component of $\Phi_M$ is $h_S (M)$.
\end{lem}

\begin{proof}
Recall that $\Phi$ is the subspace of $k[M]^*$ which annihilates $\Theta$. Therefore, its dimension in degree $S$ is $\dim_S k[M] / \Theta = \dim_S R_M = h_S (M)$. 
\end{proof}

We now proceed to the main proof. To show surjectivity, it suffices to find preimages under $\pi$ for the $h_S (B)$ members of a basis of $\Phi_B$. We start with the case $S = [r]$. Let $C$ be an essential full chain in $\Delta(B)$ with corresponding string $b_1 b_2 \dots b_{r+1}$. (By definition of essentiality, this string is completely descending.) Define
\[
f_C = \sum_{\sigma \in S_{r+1}} \sgn(\sigma) \epsilon_{C_\sigma},
\]
where $S_k$ is the symmetric group on $k$ elements and $C_\sigma$ is the full $B$-chain
\[
0_B \varsubsetneq \phi_B(\{b_{\sigma(1)}\}) \varsubsetneq \phi_B(\{b_{\sigma(1)}, b_{\sigma(2)}\}) \varsubsetneq \dots \varsubsetneq \phi_B(\{b_{\sigma(1)}, b_{\sigma(2)}, \dots, b_{\sigma(r)}\}) \varsubsetneq E.
\]
(This is a full chain because $b_1$, \dots, $b_{r+1}$ is a basis.) Similarly, define
\[
g_C = \sum_{\sigma \in S_{r+1}} \sgn(\sigma) \epsilon_{D_\sigma},
\]
where $D_\sigma$ is the $A$-chain
\[
0_A \subseteq \phi_A(\{b_{\sigma(1)}\}) \subseteq \phi_A(\{b_{\sigma(1)}, b_{\sigma(2)}\}) \subseteq \dots \subseteq \phi_A(\{b_{\sigma(1)}, b_{\sigma(2)}, \dots, b_{\sigma(r)}\}) \subseteq E.
\]
Note that $\{b_1, \dots, b_{r+1}\}$ remains a basis in $A$, so $D_\sigma$ is a full chain with one flat of each rank. We also see that $\phi_B (D_\sigma) = C_\sigma$, since for an independent set $I$, $\phi_M (I)$ is the set of elements that form a dependent set when added to $I$, and therefore $\phi_A (I) \subseteq \phi_B (I)$. This also shows that each flat of $D_\sigma$ is in $\FF(A')$. Thus $D_\sigma$ is a full $A'$-chain, and $\pi(\epsilon_{D_\sigma}) = \epsilon_{C_\sigma}$ and $\pi(g_C) = f_C$.

Next we show that the $f_C$, taken over all essential $C$, lie in $\Phi_B$, and that the $g_C$ lie in $\Phi_{A'}$. Fix an essential full chain $C$ with associated string $b_1 \dots b_{r+1}$. Given $\sigma \in S_{r+1}$, there is at least one $\sigma'$ such that $C_{\sigma'}$ differs from $C_{\sigma}$ in rank $i$ only, namely $\sigma \circ (i \ i+1)$. Now suppose that for some $\sigma' \in S_{r+1}$, $C_{\sigma'}$ differs from $C_{\sigma}$ in rank $i$ only. Then if if the rank $i-1$ and $i+1$ flats of $C_\sigma$ are $F_{i-1} = \phi_B (H)$ and $F_{i+1} = \phi_B (H \cup \{b, b'\})$ respectively, where $H, \{b, b'\} \subseteq \{b_1, \dots, b_{r+1}\}$, then there are exactly two possibilities for $F_k$, namely $\phi_B (H \cup \{b\})$ and $\phi_B (H \cup \{b'\})$. If $C_\sigma$ contains one of these two, then $C_{\sigma'}$ must contain the other one. That is, there is only one $\sigma'$ satisfying the description. Then $\sigma$ and $\sigma'$ differ by a transposition, so $\epsilon_{C_\sigma}$ and $\epsilon_{C_{\sigma'}}$ will have opposite signs in the expression for $f_C$. As a result, the condition that, for all $i \in S$ and all chains $C$ with fl $C = S\setminus i$, we have $\sum_{D \supseteq C} b_D = 0$, is satisfied, and by Prop.~\ref{phicond}, $f_C \in \Phi_B$. By the exact same argument, $g_C \in \Phi_{A'}$.

Next we show that the $f_C$ are linearly independent. To do this, we first observe that $C_\sigma \leq C$ in the lexicographic order for all $\sigma$. This is because $b_1 > b_2 > \dots > b_{r+1}$, so if the first rank in which $C_\sigma$ and $C$ differ is $k$, then $C_\sigma$'s flat of that rank contains an element less than $b_k$, and therefore comes lexicographically before $C$'s flat of rank $k$. This means that the matrix whose rows and columns are indexed by full chains of $\Delta(B)$ in lexicographic order, with the entry in row $C$, column $C'$ being the coefficient of $\epsilon_{C'}$ in $f_C$, is lower triangular, and all nonzero rows (i.e.\ those corresponding to essential $C$) have a 1 on the diagonal. Therefore these nonzero rows, hence the $f_C$ themselves, are linearly independent.

Finally, we note that the number of essential full chains has already been shown to be $h_{[r]} (B)$, which is also the dimension of $\Phi_B$ in degree $[r]$. Therefore, the $f_C$ form a basis for $\Phi_B$ in this degree. This completes the proof of the surjectivity of $\pi$ in degree $[r]$.

Now let $S$ be an arbitrary subset of $[r]$. Choose $C$ from among the full chains of $\Delta(B)$ that have descent set $S$, i.e.\ minimal completions of essential chains of flag $S$. Let $b_1 \dots b_{r+1}$ be the corresponding string. Define $C_\sigma$ and $D_\sigma$ as before, and let $\nu$ restrict a chain to the ranks in $S$. Set
\[
f_C = \sum_{\sigma \in H} \sgn(\sigma) \epsilon_{\nu(C_\sigma)},
\]
where $H$ is the subgroup of $S_{r+1}$ generated by the transpositions $\{(i \ i+1) \mid i \in S\}$. Analogously, set
\[
g_C = \sum_{\sigma \in H} \sgn(\sigma) \epsilon_{\nu(D_\sigma)}.
\]
As before, we have $g_C \in k[A']^*$ and $\pi(g_C) = f_C$. Now if $i \in S$, then $\sigma \in H$ if and only if $\sigma \circ (i \ i+1) \in H$, so a term corresponding to
\[
0_B \varsubsetneq \phi_B(\{b_{\sigma(1)}\}) \varsubsetneq \dots \varsubsetneq \phi_B(\{b_{\sigma(1)}, \dots, b_{\sigma(i-1)}, b_{\sigma(i)}\}) \varsubsetneq \dots \varsubsetneq E
\]
appears in $f_C$ if and only if one corresponding to
\[
0_B \varsubsetneq \phi_B(\{b_{\sigma(1)}\}) \varsubsetneq \dots \varsubsetneq \phi_B(\{b_{\sigma(1)}, \dots, b_{\sigma(i-1)}, b_{\sigma(i+1)}\}) \varsubsetneq \dots \varsubsetneq E
\]
(that is, a chain differing from $C_\sigma$ only in rank $i$) appears with opposite sign. That is, once again the Prop.~\ref{phicond} condition is satisfied for $f_C$ to be in $\Phi_B$ and $g_C$ to be in $\Phi_{A'}$.

To show that the $f_C$ are linearly independent, it suffices to check that $\nu(C_\sigma) \leq \nu(C)$ in lexicographic order for $\sigma \in H$. First we note that if $i \notin S$, then for all $\sigma \in H$, the rank $i$ flat in $C_\sigma$ is $\phi_B (b_1, \dots, b_i)$. Thus if $C_\sigma$ and $C$ differ, it must be at a rank in $S$. Let $j$ be the smallest rank where they differ, $i$ the largest element of $[r] \setminus S$ with $i < j$ (or 0 if no such element exists), and $i'$ the smallest element of $[r] \setminus S$ with $i' > j$ (or $r+1$ if no such element exists). By our choice of $C$, we have $b_{i+1} > \dots > b_{i'}$. Therefore the rank $j$ flat in $C_\sigma$ (as well as in $\nu(C_\sigma)$) is $\phi_B (\{b_1, \dots, b_i, b_{\sigma(i+1)}, \dots, b_{\sigma(j)}\})$. Since $\sigma$ takes $\{i+1, \dots , i'\}$ to itself, there must be some $i+1 \leq t \leq j$ with $j+1 \leq \sigma(t) \leq k'$, meaning the $j$-th flat of $\nu(C_\sigma)$ is less than the $j$-th flat of $\nu(C)$ in lexicographic order.

We have thus demonstrated $h_S (B)$ linearly independent elements of $(\Phi_B)_S$, a vector space of dimension $h_S (B)$; they are therefore a basis. For each one, we have found a $g_C \in (\Phi_{A'})_S$ with $\pi(g_C) = f_C$, therefore $\pi$ is surjective in all degrees.
\end{proof}

\begin{cor}
If $A \to B$ is a rank-preserving weak map of matroids, with the rings $R_{A'}$ and $R_B$, as well as the map $\bar{\psi}: R_B \to R_{A'}$, defined as before, then $\bar{\psi}$ is injective.
\end{cor}

\begin{proof}
This follows formally from the fact that the dual map $\pi : \Theta_{A'} \to \Theta_B$ is surjective.
\end{proof}


Corollary~\ref{cor:reduction} thus implies the desired result:

\begin{cor}
If $A \to B$ is a rank-preserving weak map of matroids, whose order complexes have flag $h$-numbers $h_S$, then $h_S (A) \geq h_S (B)$ for all $S$.
\end{cor}


By combining this result with Prop.~\ref{fh-vectors} and Thm.~\ref{smineq}, we summarize our conclusions as follows:

\begin{theorem} \label{finalthm}
Let $A \to B$ be a weak map of matroids, with $\rk(B) = r+1$. Let $h_i$ and $f_i$ be the standard $h$- and $f$-numbers of their lattice of flats. Then,

\begin{enumerate}
\item $h_i (A) \geq h_i (B)$ for all $0 \leq i \leq r$.

\item $f_i (A) \geq f_i (B)$ for all $0 \leq i \leq r$.
\end{enumerate}

If $\rk(A)$ is also $r+1$, then, letting $h_S$ and $f_S$ be the flag $h$- and $f$-numbers,

\begin{enumerate}[resume]
\item $h_S (A) \geq h_S (B)$ for all $S \subseteq [r]$.

\item $f_S (A) \geq f_S (B)$ for all $S \subseteq [r]$.
\end{enumerate}
\end{theorem}

\section*{Acknowledgments}

The authors would like to thank Spencer Backman, Sara Billey, Isabella Novik, and Nicholas Proudfoot for their helpful discussions and comments.


\begin{thebibliography}{9}

\bibitem{AHK} Karim Adiprasito, June Huh, and Eric Katz, Hodge theory for combinatorial geometries, \emph{Ann. Math.} \textbf{188} (2018), 381--452.

\bibitem{B1} Anders Bj\"{o}rner, Shellable and Cohen-Macaulay partially ordered sets, \emph{Trans. Amer. Math. Soc.}, \textbf{260} (1980), 159--183. 

\bibitem{B2} Anders Bj\"{o}rner, Homology and shellability of matroids and geometric lattices, in N.L. White, ed., \emph{Matroid Applications}, 1992, 226–283. 

\bibitem{BHMPW} Tom Braden, June Huh, Jacob P. Matherne, Nicholas Proudfoot, and Botong Wang, Singular Hodge theory for combinatorial geometries, arXiv:2010.06088.

\bibitem{FS} Luis Ferroni and Benjamin Schr\"{o}ter, Valuative invariants for large classes of matroids, arXiv:2208.04893 (2022).

\bibitem{L} Dean Lucas, Weak maps of combinatorial geometries, \emph{Trans. Amer. Math. Soc.}, \textbf{206} (1975), 247--278.

\bibitem{NS} Kathryn Nyman and Ed Swartz, Inequalities for the $h$-vectors and flag $h$-vectors of geometric lattices, \emph{Disc. Comput. Geom.}, \textbf{32} (2004), 533–-548. 

\bibitem{Ox} James Oxley, \emph{Matroid Theory}, Oxford Sci. Publ., The Clarendon Press, Oxford University Press, New York, 1992.

\bibitem{S0} Richard P. Stanley, Balanced Cohen-Macaulay complexes, \emph{Trans. Amer. Math. Soc.}. \textbf{249} (1979), 139--157. 

\bibitem{S1} Richard P. Stanley, A monotonicity property of $h$-vectors and $h^*$-vectors, \emph{Eur. J. of Combin.}, \textbf{14} (1993), 251-–258. 

\bibitem{S2} Richard P. Stanley, \emph{Enumerative Combinatorics}, Cambridge Univ. Press, 2012.

\bibitem{Stu} Bernd Sturmfels, On the matroid stratification of Grassmann varieties, specialization of coordinates, and a problem of N. White, \emph{Adv. Math.} \textbf{75} (1989), 202--211.

\bibitem{Wh} Neil White, \emph{Theory of Matroids}, Cambridge Univ. Press, 1986.
	
\end{thebibliography}
\end{document}